\documentclass[12pt,reqno,a4paper]{amsart}
\usepackage{amssymb,amscd, hyperref, epsf}

\begin{document}
\def\bgamma{\boldsymbol{\gamma}}
\let\kappa=\varkappa
\let\epsilon=\varepsilon
\let\phi=\varphi
\let\p\partial

\def\Z{\mathbb Z}
\def\R{\mathbb R}
\def\C{\mathbb C}
\def\Q{\mathbb Q}
\def\P{\mathbb P}
\def\N{\mathbb N}
\def\L{\mathbb L}
\def\HH{\mathrm{H}}
\def\ss{X}

\def\conj{\overline}
\def\Beta{\mathrm{B}}
\def\const{\mathrm{const}}
\def\ov{\overline}
\def\wt{\widetilde}
\def\wh{\widehat}

\renewcommand{\Im}{\mathop{\mathrm{Im}}\nolimits}
\renewcommand{\Re}{\mathop{\mathrm{Re}}\nolimits}
\newcommand{\codim}{\mathop{\mathrm{codim}}\nolimits}
\newcommand{\id}{\mathop{\mathrm{id}}\nolimits}
\newcommand{\Aut}{\mathop{\mathrm{Aut}}\nolimits}
\newcommand{\lk}{\mathop{\mathrm{lk}}\nolimits}
\newcommand{\sign}{\mathop{\mathrm{sign}}\nolimits}
\newcommand{\pt}{\mathop{\mathrm{pt}}\nolimits}
\newcommand{\rk}{\mathop{\mathrm{rk}}\nolimits}
\newcommand{\SKY}{\mathop{\mathrm{SKY}}\nolimits}
\newcommand{\st}{\mathop{\mathrm{st}}\nolimits}
\def\Jet{{\mathcal J}}
\def\FC{{\mathrm{FCrit}}}
\def\sS{{\mathcal S}}
\def\lcan{\lambda_{\mathrm{can}}}
\def\ocan{\omega_{\mathrm{can}}}

\renewcommand{\mod}{\mathrel{\mathrm{mod}}}
\def\ds{\displaystyle}

\newtheorem{mainthm}{Theorem}
\newtheorem{thm}{Theorem}[section]
\newtheorem{lem}[thm]{Lemma}
\newtheorem{prop}[thm]{Proposition}
\newtheorem{cor}[thm]{Corollary}
\newtheorem{conjecture}[thm]{Conjecture}

\theoremstyle{definition}
\newtheorem{exm}[thm]{Example}
\newtheorem{rem}[thm]{Remark}
\newtheorem{df}[thm]{Definition}

\renewcommand{\thesubsection}{\arabic{subsection}}
\numberwithin{equation}{subsection}

\title{Causality and Legendrian linking for higher dimensional spacetimes }
\author[Chernov]{Vladimir Chernov}
\thanks{This work was partially supported by a grant from the Simons Foundation
(\# 513272 to Vladimir Chernov). The author is thankful to Stefan Nemirovski for many useful discussions.}
\address{Department of Mathematics, 6188 Kemeny Hall,
Dartmouth College, Hanover, NH 03755, USA}
\email{Vladimir.Chernov@dartmouth.edu}

\begin{abstract}
Let $(\ss^{m+1}, g)$ be an $(m+1)$-dimensional globally hyperbolic spacetime with Cauchy surface $M^m$, and let $\widetilde M^m$ be the universal cover of the Cauchy surface.
Let $\mathcal N_{\ss}$ be the contact manifold of all future directed unparameterized light rays in $\ss$ that we identify with the spherical cotangent bundle $ST^*M.$
Jointly with Stefan Nemirovski we showed that when $\widetilde M^m$ is {\bf not\/} a compact manifold, then two points $x, y\in \ss$ are causally related if and only if the Legendrian spheres $\mathfrak S_x, \mathfrak S_y$ of all light rays through $x$ and $y$ are linked in $\mathcal N_{\ss}.$

In this short note we use the contact Bott-Samelson theorem of Frauenfelder, Labrousse and Schlenk to show that the same statement is true for all $\ss$ for which the integral cohomology ring of a closed $\widetilde M$ is {\bf not} the one of the CROSS (compact rank one symmetric space).

If $M$ admits a Riemann metric $\ov g$, a point $x$ and a number $\ell>0$ such that all unit speed geodesics starting from $x$ return back to $x$ in time $\ell$, then $(M, \ov g)$ is called a $Y^x_{\ell}$ manifold.  Jointly with Stefan Nemirovski  we observed that causality in $(M\times \R, \ov g\oplus -t^2)$ is {\bf  not} equivalent to Legendrian linking. Every $Y^x_{\ell}$-Riemann manifold has compact universal cover and its integral cohomology ring is the one of a CROSS. So we conjecture that Legendrian linking is equivalent to causality if and only if  one can {\bf not} put a $Y^x_{\ell}$ Riemann metric on a Cauchy surface $M.$
\end{abstract}

\maketitle

All manifolds, maps etc.~are assumed to be smooth unless the opposite is explicitly stated,
and the word {\it smooth\/} means $C^{\infty}$.

\section{Introduction}
Let $M$ be a not necessarily orientable, connected manifold of dimension $m\ge 2$
and let $\pi_M:ST^*M\to M$ be its spherical cotangent bundle.
The manifold $ST^*M$ carries a canonical co-oriented contact structure.
An isotopy $\{L_t\}_{t\in[0,1]}$ of Legendrian submanifolds in a co-oriented
contact manifold is called respectively {\it positive, non-negative \/} if it can be parameterised
in such a way that the tangent vectors of all the trajectories of individual points
lie in respectively positive, non-negative tangent half-spaces defined by the contact structure.

For the introduction of basic notions from Lorentz geometry we follow our paper~\cite{ChernovNemirovskiGAFA}.

Let $(\ss^{m+1}, g)$ be an $(m+1)$-dimensional Lorentz manifold and $x\in \ss$. A nonzero ${\bf v}\in T_x\ss$ 
is called {\it timelike, spacelike, non-spacelike (causal) or null (lightlike)\/} if $g({\bf v}, {\bf v})$ is respectively 
negative, positive, non-positive or zero. An piecewise smooth curve is timelike if all of its velocity 
vectors are timelike. Null and non-spacelike curves are defined similarly. The Lorentz manifold $(\ss, g)$ has a unique 
Levi-Cevita connection, see for example~\cite[page 22]{BeemEhrlichEasley}, so we can talk about timelike and null geodesics. A submanifold $M\subset \ss$ is {\it spacelike} if $g$ restricted to $TM$ is a 
Riemann metric.

All non-spacelike vectors in $T_x\ss$ form a cone consisting of two hemicones, and a continuous with respect to $x\in \ss$ choice of one of the two hemicones (if it exists) is called the {\it time orientation\/} of
$(\ss, g)$. The vectors from the chosen hemicones are called {\it future directed.\/}
A time oriented Lorentz manifold is called a {\it spacetime\/} and its points are called {\it events.\/}

For $x$ in a spacetime $(\ss, g)$ its {\it causal future\/} $J^+(x)\subset \ss$ 
is the set of all $y\in \ss$ that can be reached by a 
future pointing causal curve from $x.$ 
The causal past $J^-(x)$ of the event $x\in \ss$ is defined similarly.

Two events  $x,y$ are said to be {\it causally related\/} if $x\in J^+(y)$ or $y\in J^+(x).$

A spacetime is said to be {\it globally hyperbolic\/} if $J^+(x)\cap J^-(y)$ is compact for every $x,y\in \ss$
and if it is {\it causal,\/} i.e.~it has no closed non-spacelike curves. The classical definition 
of global hyperbolicty requires $(\ss,g)$ to be strongly causal rather than just causal, but these two definitions are equivalent, see Bernal and Sanchez~\cite[Theorem 3.2]{BernalSanchezCausal}.

A {\it Cauchy surface\/} in $(\ss, g)$ is a subset such that every inextendible nonspacelike curve 
$\gamma(t)$ intersects it at exactly one value of $t.$ A classical result is that $(\ss, g)$ is globally 
hyperbolic if and only if it has a Cauchy surface, see~\cite[pages 211-212]{HawkingEllis}. 
Geroch~\cite{Geroch} proved that every globally hyperbolic $(\ss, g)$ is homemorphic to a product of $\R$ 
and a Cauchy surface. Bernal and Sanchez~\cite[Theorem 1]{BernalSanchez},~\cite[Theorem
1.1]{BernalSanchezMetricSplitting},~\cite[Theorem 1.2]{BernalSanchezFurther} proved that every globally 
hyperbolic $(\ss^{m+1}, g)$ has a smooth spacelike Cauchy surface $M^m$ and that moreover
for every smooth spacelike Cauchy surface $M$ there is a
diffeomorphism  $h:M\times \R\to \ss$ such that 
\begin{description}
\item[a] $h(M\times t)$ is a smooth spacelike Cauchy surface
for all $t$, 
\item[b] $h(x\times \R)$ is a future directed timelike curve for all $x\in M$, and finally 
\item[c] $h(M\times 0)=M$ with $h|_{M\times 0}:M\to M$ being the identity map.
\end{description}

For a spacetime $\ss$ we consider its space of light rays $\mathfrak N=\mathfrak N_X$.
By definition, a point $\bgamma\in\mathfrak N$ is an equivalence class of inextendible
future-directed null geodesics up to an orientation preserving affine reparametrization.

A seminal observation of Penrose and Low~\cite{PR2, Lo1, Lo2} is that the space $\mathfrak N$
has a canonical structure of a contact manifold (see also \cite{NatarioTod, KT, BIL, ChernovNemirovskiRedShift}). 
A contact form $\alpha_M$ on $\mathfrak N$ 
defining that contact structure can be associated to any smooth spacelike Cauchy 
surface~$M\subset X$. Namely, consider the map 
$$
\iota_M:\mathfrak N_X \mathrel{\lhook\joinrel\longrightarrow} T^*M
$$
taking $\bgamma\in\mathfrak N$ represented by a null geodesic~$\gamma\subset X$
to the $1$-form on $M$ at the point $x=\gamma\cap M$ 
collinear to ${\langle\dot\gamma(x),\cdot\,\rangle|}_M$
and having unit length with respect to the induced Riemann metric on~$M$.
This map identifies $\mathfrak N$ with the unit cosphere bundle $ST^*M$
of the Riemannian manifold $M.$
Then 
$$
\alpha_M := \iota_M^*\,\lambda_{\mathrm{can}},
$$
where $\lambda_{\mathrm{can}}=\sum p_kdq^k$ is the canonical Liouville $1$-form on $T^*M$.

\begin{rem}[Bott-Samelson type result of Frauenfelder, Labrousse and Schlenk and its strengthening]\label{BSFLS}
The contact Bott-Samelson~\cite{BottSamelson} type result of Frauenfelder, Labrousse and Schlenk \cite[Theorem 1.13]{FLS} says that if there is a positive Legendrian isotopy of a fiber $S_x$ of $ST^*M$ to itself,  then the universal cover $\widetilde M$ of $M$ is compact and has the integral cohomology ring of a CROSS. (This result answers our question with Nemirovski~\cite[Example 8.3]{ChernovNemirovskiGT} and compactness of $\widetilde M$ was first aproved in~\cite[Corollary 8.1]{ChernovNemirovskiGT}.)

Our result with Nemirovski~\cite[Proposition 4.5]{ChernovNemirovskiJSG} says that if there is a non-constant non-negative Legendrian isotopy of $S_x$ to itself, then there is a positive Legendrian isotopy of $S_x$ to itself. (Note that this positive Legendrian isotopy generally is not a perturbation of the non-negative non-constant Legendrian isotopy that was assumed.) 

So we can somewhat strengthen~\cite[Theorem 1.13]{FLS} to say that if there is a non-constant non-negative Legendrian isotopy of $S_x$ to itself,  then the universal cover $\widetilde M$ of $M$ is compact and has the integral cohomology ring of a CROSS, and in this work we will have to use the strengthened version of the contact Bott-Samelson theorem.
\end{rem}

\section{Main Results}
Let $(\ss,g)$ be a globally hyperbolic spacetime with Cauchy surface $M.$
For a point $x\in \ss$ we denote by $\mathfrak S_x\subset \mathfrak N$ the Legendrian sphere of all (unparameterized, future directed) light rays passing through $x.$

For two causally unrelated points $x,y\in \ss$ the Legendrian link $(\mathfrak S_x, \mathfrak S_y)$ in $\mathcal N$ does not depend on the choice of the causally unrelated points. Under the identification $\mathcal N=ST^*M$  this link is Legendrian isotopic to the link of sphere-fibers over two points of some (and then any) spacelike Cauchy surface $M,$ see~\cite[Theorem 8]{ChernovRudyak},~\cite[Lemma 4.3]{ChernovNemirovskiGAFA} and~\cite{NatarioTod}. We call a Legendrian link {\it trivial\/} if it is isotopic to such a link.

In~\cite[Theorem A]{ChernovNemirovskiGAFA} and~\cite[Theorem 10.4]{ChernovNemirovskiGT} we proved the following result. Assume that the universal cover $\widetilde M$ of a Cauchy surface of $M$ of $\ss$ is not compact and events $x,y\in\ss$ are causally related. Then the Legendrian link $(\mathfrak S_x, \mathfrak S_y)$ is nontrivial. 

In the case where 
$M=\R^3$ this proved the Legendrian Low conjecture of Natario and Tod~\cite{NatarioTod}. The question to explore relations between causality  and linking was motivated by the observations of Low~\cite{Lo1, Lo2} and appeared on Arnold's problem lists
as a problem communicated by Penrose~\cite[Problem~8]{ArnoldProblem},
\cite[Problem 1998-21]{ArnoldProblemBook}.

In this work we prove the following Theorem.

\begin{thm}\label{maintheorem1}
Assume that two events $x,y$ in a globally hyperbolic spacetime $\ss$ are causally related and the universal cover $\widetilde M$ of a Cauchy surface $M$ of $\ss$ is compact but does {\bf not} have integral cohomology ring of a compact rank one symmetric space (CROSS). Then the Legendrian link $(\mathfrak S_x, \mathfrak S_y)$ is nontrivial.
\end{thm}

\begin{proof}
The beginning of the proof follows the one of our~\cite[Theorem A]{ChernovNemirovskiGAFA}. 

If $x,y$ are on the same null geodesic then the Legendrian link $(\mathfrak S_x, \mathfrak S_y)$ has a double point and hence is singular. So we do not have to consider this case.

Without the loss of generality we can assume that $y\in J^+(x)$ and hence $\mathfrak S_y$ is connected to $\mathfrak S_x$ by a non-negative Legendrian isotopy,~see~\cite[Proposition 4.2]{ChernovNemirovskiGAFA}. 

Suppose that $\mathfrak S_x$ and $\mathfrak S_y$ are Legendrian unlinked, i.\,e.,
that the link $\mathfrak S_x, \mathfrak S_y$ is Legendrian isotopic to
the link $F\sqcup F'$ formed by  two different fibers of $ST^*M$.
By the Legendrian isotopy extension theorem,
there exists a contactomorphism $\Psi:ST^*M\to ST^*M$ such that
$\Psi(\mathfrak S_x\sqcup\mathfrak S_y)=F\sqcup F'$.
Hence, we get a non-negative Legendrian isotopy
connecting two different fibres 
$F$ and $F'$ of $ST^*M$. But the Legendrian link $F\sqcup F'$ is symmetric, i.\,e., it is Legendrian isotopic to $F'\sqcup F$. Hence there exists a contactomorphism of $ST^*M$ exchanging the two link components and we also have a non-negative Legendrian isotopy connecting $F'$ to $F.$

Composing the non-negative Legendrian isotopy from $F$ to $F'$ and from $F'$ to $F$ we get a non-constant non-negative Legendrian isotopy from a fiber of $ST^*M$ to itself.

Finally Remark~\ref{BSFLS} says that $\widetilde M$ has the integral cohomology ring of a CROSS. This contradicts to our assumptions. 

\end{proof}

\begin{rem}
Let $(M, \overline g)$ be a Riemann manifold having a point $x$ and a positive number $\ell$ such that all the unit speed geodesics from $x$ return back to $x$ in time $\ell$, then $(M, \overline g)$ is called a $Y^x_{\ell}$ Riemann manifold. The result of B\'erard-Bergery~\cite[Theorem 7.37]{Besse},~\cite{BerardBergery} says that the universal cover $\widetilde M$ in this case is compact and the  rational cohomology ring of $\widetilde M$ has exactly one generator. Clearly the co-geodesic flow on $M$ gives a positive Legendrian isotopy of the Legendrian sphere fiber of $ST^*M$ over $x$ to itself. So the Bott-Samelson~\cite{BottSamelson} type result of Frauenfelder, Labrousse and Schlenk~\cite[Theorem 1.13]{FLS} says more, namely that $\widetilde M$ has the integral cohomology ring of a CROSS.

In~\cite[Example 10.5]{ChernovNemirovskiGT} we observed that when $(M, \overline g)$ is a $Y^x_{\ell}$ Riemann manifold, then causality in the globally hyperbolic $(M\times\R, \overline g\oplus -dt^2)$ is {\bf not} equivalent to Legendrian linking of sphere of light rays. In particular, causality is not equivalent to Legendrian linking in the case where $(M, \overline g)$ is a
metric quotient of the standard sphere. The Thurston Elliptization Conjecture (solved as the part of the Thurston Geometrization Conjecture) by Perelman~\cite{Perelman1, Perelman2, Perelman3} says that every $3$-dimensional $M$ whose universal cover is compact (and hence is  the sphere by the Poincare conjecture) is a metric quotient of the sphere. {\bf  So the results of Theorem~\ref{maintheorem1} are new and interesting only for spacetimes of dimension greater than four. \/}
\end{rem}

It is not currently known whether every compact simply connected manifold whose integer cohomology ring is the one of a CROSS admits a $Y^x_{\ell}$ Riemann metric. So we make the following Conjecture.

\begin{conjecture}
Assume that a globally hyperbolic spacetime $\ss$ is such that one can {\bf not} put a $Y^x_{\ell}$ Riemann metric on its Cauchy surface $M.$ Then two events $x,y\in\ss$ are causally related if and only if the Legendrian link $(\mathfrak S_x, \mathfrak S_y)$ is non-trivial.
\end{conjecture}

In~\cite[Theorem C]{ChernovNemirovskiGAFA} we proved the following result. Assume that a Cauchy surface $M$ has a cover diffeomorphic to an open domain in $\R^m$ and two events $x,y\in \ss$ are such that 
$y \in J^+(x)$ but $x,y$ do not belong to a common light geodesic. Then the Legendrian links  $\mathfrak S_x\sqcup\mathfrak S_y$
and $\mathfrak S_y\sqcup\mathfrak S_x$ are not Legendrian isotopic. (In this case there is a non-negative Legendrian isotopy from $\mathfrak S_y$ to $\mathfrak S_x$, but there is no such non-negative Legendrian isotopy from $\mathfrak S_x$ to $\mathfrak S_y$.) One of the key ingredients in the proof is that for such $M$ there is no non-constant non-negative Legendrian isotopy of a fiber of $ST^*M$ to itself. 
In~\cite[Corollary 8.1, Remark 8.2]{ChernovNemirovskiGT} we showed that for the case where the universal cover $\widetilde M$ of $M$ is not a compact manifold there is no non-negative Legendrian isotopy of a fiber of $ST^*M$ to itself. So the result and the proof of~\cite[Theorem C]{ChernovNemirovskiGAFA} immediately generalize to this case.

In this work we use the Bott-Samelson~\cite{BottSamelson} type Theorem of Frauenfelder, Labrousse and Schlenk \cite[Theorem 1.13]{FLS} to get the following result.

\begin{thm}\label{maintheorem2}
Assume that two events $x,y$ in a globally hyperbolic spacetime $\ss$ are causally related and the universal cover $\widetilde M$ of a Cauchy surface $M$ of $\ss$ is compact but does {\bf not} have integral cohomology ring of a compact rank one symmetric space (CROSS). Then the Legendrian link $(\mathfrak S_x, \mathfrak S_y)$ is not Legendrian isotopic to 
$(\mathfrak S_y, \mathfrak S_x)$. 
\end{thm}

\begin{proof}
The beginning of the proof follows the one of our~\cite[Theorem C]{ChernovNemirovskiGAFA}. 

Suppose that the links $\mathfrak S_x, \mathfrak S_y$
and $\mathfrak S_y, \mathfrak S_x$ are Legendrian isotopic.
By the Legendrian isotopy extension theorem, there exists an auto contactomorphism $\Psi$
such that $\Psi(\mathfrak S_x, \mathfrak S_y)=(\mathfrak S_y, \mathfrak S_x)$.
Without loss of generality we assume that $y$ is in the causal past of $x$.
Let $\{\Lambda_t\}_{t\in[0,1]}$ be a non-negative Legendrian isotopy
in $ST^*M$ connecting $\mathfrak S_x$ to $\mathfrak S_y$
provided by~\cite[Proposition 4.2]{ChernovNemirovskiGAFA}. Then $\{\Psi(\Lambda_t)\}_{t\in[0,1]}$ is
a non-negative Legendrian isotopy connecting $\mathfrak S_y$ to~$\mathfrak S_x$.
Composing these two isotopies, we obtain a non-constant non-negative
Legendrian isotopy connecting $\mathfrak S_x$ to itself.

Recall now that  $\mathfrak S_x$ is Legendrian isotopic to a fiber of $ST^*M$.
By the Legendrian isotopy extension theorem, there exists a contactomorphism $\Phi$
taking $\mathfrak S_x$ to that fiber. The Legendrian isotopy connecting $\mathfrak S_x$ to itself
constructed above is taken by $\Phi$ to a non-constant non-negative Legendrian isotopy connecting
the fiber to itself. 

Finally Remark~\ref{BSFLS} says that $\widetilde M$ has the integral cohomology ring of a CROSS. This contradicts to our assumptions.

\end{proof}

\end{document}